\documentclass[a4paper,12pt,oneside,headsepline]{scrartcl}

\usepackage{ifdraft}

\usepackage[utf8]{inputenc}
\usepackage[T1]{fontenc}

\usepackage{lmodern}
\usepackage{fourier}
\usepackage{amssymb, amsfonts}

\usepackage{mathrsfs} %
\usepackage{dsfont}
\usepackage[usenames,dvipsnames]{xcolor}%
\usepackage{lettrine}
\usepackage{url}

\usepackage{stmaryrd}

\usepackage[english]{babel}

\usepackage[draft=false]{scrlayer-scrpage}
\usepackage{tocstyle}
\usetocstyle{KOMAlike} 

\usepackage{wrapfig}
\usepackage{footmisc}
\usepackage{needspace}

\usepackage{amsmath}
\usepackage[thmmarks, amsmath, amsthm]{ntheorem}

\usepackage[style=numeric-comp,giveninits,url=false,sortcites=true,maxbibnames=99]{biblatex}
\bibliography{bibliography.bib}

\ifdraft{{}}
  {\usepackage[pdftex,
               pdfborder={0 0 0}, colorlinks=true,
               linkcolor=BrickRed, citecolor=ForestGreen, urlcolor=RoyalBlue]{hyperref}}

\ifdraft{%
\usepackage[colorinlistoftodos]{todonotes}
\usepackage{lineno}
}{}

\usepackage{booktabs}
\usepackage[inline]{enumitem}
\usepackage{float}
\usepackage{graphicx}
\usepackage{multicol}
\usepackage{scrtime}
\usepackage{tikz}
\usetikzlibrary{matrix,arrows,positioning}

\ifdraft{\date{\today}}{\date{}}

\ifdraft{%
\linenumbers
\newcommand*\patchAmsMathEnvironmentForLineno[1]{%
  \expandafter\let\csname old#1\expandafter\endcsname\csname #1\endcsname
  \expandafter\let\csname oldend#1\expandafter\endcsname\csname end#1\endcsname
  \renewenvironment{#1}%
     {\linenomath\csname old#1\endcsname}%
     {\csname oldend#1\endcsname\endlinenomath}}%
\newcommand*\patchBothAmsMathEnvironmentsForLineno[1]{%
  \patchAmsMathEnvironmentForLineno{#1}%
  \patchAmsMathEnvironmentForLineno{#1*}}%
\AtBeginDocument{%
\patchBothAmsMathEnvironmentsForLineno{equation}%
\patchBothAmsMathEnvironmentsForLineno{align}%
\patchBothAmsMathEnvironmentsForLineno{flalign}%
\patchBothAmsMathEnvironmentsForLineno{alignat}%
\patchBothAmsMathEnvironmentsForLineno{gather}%
\patchBothAmsMathEnvironmentsForLineno{multline}%
}}

\KOMAoptions{DIV=10}

\clearscrheadings
\automark[section]{subsection}

\renewcommand{\subsectionmark}[1]{}

\lohead{{\small \headertitle}}
\rohead{{\small \headerauthors}}

\RedeclareSectionCommand[%
  font=\Large\sffamily\bfseries,%
  beforeskip=1\baselineskip,%
  afterskip=0.5\baselineskip,%
  indent=0em%
  ]{section}

\RedeclareSectionCommands[%
  font=\normalfont\bfseries,%
  afterskip=-1em%
  ]{subsection,subsubsection}

\RedeclareSectionCommands[%
  font=\normalfont\itshape,%
  afterskip=-1em,%
  indent=0pt,
  ]{paragraph}

\AtEveryBibitem{\clearfield{isbn}}
\AtEveryBibitem{\clearlist{language}}
\AtEveryBibitem{\clearfield{pages}}

\newenvironment{enumerateroman}{
\begin{enumerate}[label=(\roman*), leftmargin=0pt,labelindent=2em,itemindent=!]
}{
\end{enumerate}
}

\newenvironment{enumeratearabic*}{
\begin{enumerate*}[label=(\arabic*)] %
}{
\end{enumerate*}
}

\newenvironment{enumerateroman*}{
\begin{enumerate*}[label=(\roman*)] %
}{
\end{enumerate*}
}

\numberwithin{equation}{section}

\theoremnumbering{arabic}
\newtheorem{theoremcounter}{theoremcounter}[section]
\theoremnumbering{Alph}
\newtheorem{maintheoremcounter}{maintheoremcounter}

\theoremstyle{plain}

\newtheorem{corollary}[theoremcounter]{Corollary}

\newtheorem{proposition}[theoremcounter]{Proposition}
\newtheorem{theorem}[theoremcounter]{Theorem}

\theoremstyle{plain}

\newtheorem{maintheorem}[maintheoremcounter]{Theorem}

\theoremstyle{definition}

\theoremstyle{remark}

\newtheorem{remark}[theoremcounter]{Remark}

\theoremstyle{nonumberremark}

\newtheorem{remarkcomputation}{Computation}

\let\cal\undefined

\ifdraft{
 \newcommand{\texpdf}[2]{#1}
}{
 \newcommand{\texpdf}[2]{\texorpdfstring{#1}{#2}}
}

\newcommand{\tx}{\ensuremath{\text}}

\newcommand{\tbf}{\bfseries}

\newcommand{\thdash}{\nbd th}

\newcommand{\nbd}{\nobreakdash-\hspace{0pt}}

\newcommand{\cal}{\ensuremath{\mathcal}}
\renewcommand{\frak}{\ensuremath{\mathfrak}}

\newcommand{\cO}{\ensuremath{\cal{O}}}

\newcommand{\fraku}{\ensuremath{\frak{u}}}

\newcommand{\rmt}{\ensuremath{\mathrm{t}}}

\newcommand{\rmM}{\ensuremath{\mathrm{M}}}

\newcommand{\rmT}{\ensuremath{\mathrm{T}}}

\newcommand{\td}{\tilde}
\newcommand{\wtd}{\widetilde}
\newcommand{\ov}{\overline}

\newcommand{\llbrkt}{\llbracket}
\newcommand{\rrbrkt}{\rrbracket}

\newcommand{\ra}{\ensuremath{\rightarrow}}
\newcommand{\hra}{\ensuremath{\hookrightarrow}}

\newcommand{\lra}{\ensuremath{\longrightarrow}}

\newcommand{\mto}{\ensuremath{\mapsto}}
\newcommand{\lmto}{\ensuremath{\longmapsto}}

\newcommand{\amid}{\ensuremath{\mathop{\mid}}}

\newcommand{\ZZ}{\ensuremath{\mathbb{Z}}}
\newcommand{\QQ}{\ensuremath{\mathbb{Q}}}
\newcommand{\RR}{\ensuremath{\mathbb{R}}}
\newcommand{\CC}{\ensuremath{\mathbb{C}}}

\renewcommand{\Im}{\ensuremath{\mathrm{Im}}}

\newcommand{\isdiv}{\amid}
\newcommand{\nisdiv}{\ensuremath{\mathop{\nmid}}}

\renewcommand{\pmod}[1]{\ensuremath{\;(\mathrm{mod}\, #1)}}

\newenvironment{psmatrix}{\left(\begin{smallmatrix}}{\end{smallmatrix}\right)}

\newcommand{\Mat}[1]{\ensuremath{\mathrm{Mat}_{#1}}}
\newcommand{\GL}[1]{\ensuremath{\mathrm{GL}_{#1}}}

\newcommand{\SL}[1]{\ensuremath{\mathrm{SL}_{#1}}}
\newcommand{\Mp}[1]{\ensuremath{\mathrm{Mp}_{#1}}}

\newcommand{\U}[1]{\ensuremath{\mathrm{U}_{#1}}}

\newcommand{\T}{\ensuremath{\rmt}}

\renewcommand{\det}{\ensuremath{\mathrm{det}}}

\renewcommand{\ker}{\ensuremath{\mathrm{ker}}}

\newcommand{\HS}{\mathbb{H}}

\newcommand{\lcm}{\ensuremath{\mathrm{lcm}}}

\newcommand{\om}{\ensuremath{\omega}}
\newcommand{\ga}{\ensuremath{\gamma}}
\newcommand{\Ga}{\ensuremath{\Gamma}}
\newcommand{\tdGa}{\ensuremath{\wtd{\Gamma}}}

\newcommand{\GMp}[1]{\ensuremath{\mathrm{GMp}_{#1}}}

\newcommand{\Ind}{\ensuremath{\mathrm{Ind}}}

\newcommand{\rmfe}{\ensuremath{\mathrm{fe}}}
\newcommand{\rmFE}{\ensuremath{\mathrm{FE}}}

\newcommand{\headertitle}{{\normalfont%
  Congruences on square-classes for the partition function
}}
\newcommand{\headerauthors}{%
  M.~Raum%
}
\title{%
  Congruences on Square-Classes\\for the Partition Function
}
\author{%
Martin Raum%
\thanks{The author was partially supported by Vetenskapsr\aa det Grant~2015-04139 and~2019-03551.}%
}

\begin{document}

\thispagestyle{scrplain}
\begingroup
\deffootnote[1em]{1.5em}{1em}{\thefootnotemark}
\maketitle
\endgroup

{\small
\noindent
{\tbf Abstract:}
We considerably improve Ono's and Ahlgren-Ono's work on the frequent occurrence of Ra\-ma\-nu\-jan-type congruences for the partition function, and demonstrate that Ra\-ma\-nu\-jan-type congruences occur in families that are governed by square-classes. We thus elucidate for the first time an exemplary family of congruences found by Atkin-O'Brien. Our results are based on a novel framework that leverages available results on integral models of modular curves via representations of finite quotients of~$\SL{2}(\ZZ)$ or~$\Mp{1}(\ZZ)$. This framework applies to congruences of all weakly holomorphic modular forms.
\\[.35em]
\textsf{\textbf{%
  partition function%
}}%
\hspace{0.3em}{\tiny$\blacksquare$}\hspace{0.3em}%
\textsf{\textbf{%
  Ramanujan-type congruences%
}}%
\hspace{0.3em}{\tiny$\blacksquare$}\hspace{0.3em}%
\textsf{\textbf{%
  Atkin-O'Brien-type congruences%
}}
\\[0.15em]
\noindent
\textsf{\textbf{%
  MSC Primary:
  11P83%
}}
\hspace{0.3em}{\tiny$\blacksquare$}\hspace{0.3em}%
\textsf{\textbf{%
  MSC Secondary:
  05A17,11F30%
}}
\\[.35em]
}

\Needspace*{4em}
\addcontentsline{toc}{section}{Introduction}
\markright{Introduction}
\lettrine[lines=2,nindent=.2em]{\tbf T}{he} partition function records the number of ways~$p(n)$ to write a positive integer~$n$ as the sum of any nonincreasing sequence of positive integers. Ramanujan established the congruences
\begin{align*}
  p(5 n + 4) \;&\equiv\; 0 \;\pmod{5}
\\
  p(7 n + 5) \;&\equiv\; 0 \;\pmod{7}
\\
  p(11 n + 6) \;&\equiv\; 0 \;\pmod{11}
\end{align*}
and suggested further such congruences modulo powers of~$5$, $7$, and~$11$, which Atkin, Lehner, and Watson proved between the 30ies and 60ies. Ramanujan's results ignited interest in both justifications of combinatorial or physical nature and the existence of congruences beyond the then available conjectures. They brought to light a new, flourishing research area. Milestones in the field were subsequently reached by, for instance, Ahlgren, Andrews, Boylan, Dyson, Garvan, Ono, and Swinnerton-Dyer~\cite{andrews-1976,ahlgren-boylan-2003,andrews-garvan-1988,dyson-1944,atkin-swinnerton-dyer-1954}.

In the late 90ies and early 2000s, Ono and Ahlgren-Ono set the gold standard in the existence of congruences~\cite{ono-2000,ahlgren-ono-2001}: Given a prime~$\ell \ge 5$, a positive integer~$m$, and an integer~$\beta$ subjection to a restriction on the square class of $24 \beta - 1 \,\pmod{\ell}$, they showed that a positive proportion of primes~$Q \equiv -1 \,\pmod{24 \ell}$ have the property that
\begin{gather*}
  p\Big( \frac{Q^3 n + 1}{24} \Big) \;\equiv\; 0 \;\pmod{\ell^m}
\end{gather*}
for all integers~$n \equiv 1 - 24 \beta \,\pmod{24 \ell}$ and co-prime to~$Q$. Here and throughout, we adopt the convention that $p(n) = 0$ if~$n$ is not a nonnegative integer, which allows us to relax the condition on~$n$ to $n \equiv 1 - 24 \beta \,\pmod{\ell}$. This most crucially encodes the square-class of~$Q^3 n \slash 24 \,\pmod{\ell}$.

It was remarked in a review of Ahlgren-Ono's work that their result accommodates the shape of any known partition congruence. Previous work had been restricted to congruences~$\forall n \in \ZZ :\, p(A n + B) \equiv 0 \,\pmod{\ell^m}$ with $24 B \equiv 1 \pmod{\ell}$, while now the square-class condition expressed by the inequality~$\big( \frac{24 B - 1}{\ell} \big) \ne \big( \frac{-1}{\ell} \big)$ of Legendre symbols was the only remaining constraint. Radu~\cite{radu-2013} confirmed a conjecture by Ahlgren-Ono that a congruence modulo~$\ell$ of the partition function on the arithmetic progression~$A \ZZ + B$ implies both~$\ell \isdiv A$ and~$\big( \frac{24 B - 1}{\ell} \big) \ne \big( \frac{-1}{\ell} \big)$. In this sense, Ahlgren-Ono's accomplishment indeed exhausts all expected congruences of the partition function. It has not been qualitatively improved since its publication in~2001. We present an innovation in this paper that allows us to move forward.

There are, indeed, two twists to the story, which we elucidate in the present work. Firstly, the third power of~$Q$ in Ahlgren-Ono's result does not subsume the following example, which they happen to highlight in their paper as a motivation:
\begin{gather*}
  p(17303 n + 237)
\;=\;
  p(11^3 \cdot 13 \, n - 4 \cdot 11^2 + 721)
\;\equiv\;
  0
  \;\pmod{13}
\end{gather*}
for all integers~$n$, or equivalently,
\begin{gather*}
  p\Big( \frac{11^2 (11 \cdot 13 n - 96) + 1}{24} \Big)
\;\equiv\;
  0
  \;\pmod{13}
\end{gather*}
for all integers~$n$. While the square-class condition analogous to the one of Ahlgren-Ono holds for~$- 11^2 \cdot 96 \slash 24 \,\pmod{13}$, the decisive point is that the second formulation features a leading factor~$11^2$ as opposed to~$11^3$. It is therefore not covered by Ahlgren-Ono's result, where the leading factor is~$Q^3$. Until now, it was conceivable that congruences as the one in~\cite{ahlgren-ono-2001} with leading factor~$Q^2$ might not occur in equal abundance as those featuring the third power of~$Q$. Our Theorem~\ref{mainthm:abundance} settles this point and shows that they do occur as frequently.

Secondly, the above and all other congruences observed appear in families across square-classes. We have the much stronger result that~$p(11^2 n + 721) \equiv 0 \,\pmod{13}$ for all~$n \in \ZZ$ with~$\big( \frac{n}{11} \big) = \big( \frac{-1}{11} \big)$ and $\big( \frac{n}{13} \big) = \big( \frac{-1}{13} \big)$. We would like to refer to such congruences as of Atkin-O'Brien type, since to the best of our knowledge they were recorded for the first time in Theorem~10 of their joint work~\cite{atkin-obrien-1967}. In the case at hand, they entail Ramanujan-type congruences $p(17303 n + B) \equiv 0 \,\pmod{13}$ for thirty ~$B$ distinct modulo~$17303$:
\begin{align*}
  B
\;\in\;
  \big\{{}
&
  237, 358, 600, 1931, 2778, 4230, 4351, 4956, 5561, 5924, 6892, 8102, 8223, 8949, 
\\&
  9675, 10280, 11248, 11611, 12095, 12216, 12942, 13426, 13668, 14757, 14999,
\\&
  15241, 16088, 16330, 16572, 16935
  \big\}
\tx{.}
\end{align*}
Notably also Ramanujan's original congruences~$p(\ell n - \delta_\ell) \equiv 0 \,\pmod{\ell}$ for all~$n \in \ZZ$, where $\delta_\ell \in \ZZ$ is an inverse of~$-24$ modulo~$\ell$, fit the pattern of Atkin-O'Brien-type congruences, since we have~$p(n) \equiv 0 \,\pmod{\ell}$ for all~$n \in \ZZ$ with~$\big(\frac{n + \delta_\ell}{\ell} \big) = 0$ for $\ell \in \{5, 7, 11\}$.

In Theorem~\ref{mainthm:atkin-obrien} of this paper we offer a fundamentally different perspective on congruences of the partition function, showing that Ramanujan-type congruences of the form~$p(A n + B) \equiv 0 \,\pmod{\ell^m}$ for all~$n \in \ZZ$ imply Atkin-O'Brien-type congruences. In conjunction with the results of Ahlgren-Ono, this shows abundance of congruences~$p(Q^3 n - \delta_{\ell Q^3}) \equiv 0 \,\pmod{\ell^m}$, where~$n \in \ZZ$ is merely required to satisfy a square-class condition modulo~$\ell$ and be coprime to~$Q$, and~$\delta_{\ell Q^3} \in \ZZ$ is an inverse of~$-24$ modulo~$\ell Q^3$. We manage to further strengthen this significantly in Theorem~\ref{mainthm:abundance}, where we reduce the exponent of~$Q$ from three to~two while excluding for~$n$ only one of the three square-classes modulo~$Q$.

\begin{maintheorem}
\label{mainthm:atkin-obrien}
Let~$\ell \ge 5$ be a prime and $m$ a positive integer. Given a positive integer~$A$ co-prime to~$6 \ell$ and an integer~$B$, assume that $p(An + B) \equiv 0 \,\pmod{\ell^m}$ for all $n \in \ZZ$. Let~$A_{\mathrm{sf}}$ be the largest square-free divisor of~$A$, $A' := \gcd\big( A, A_{\mathrm{sf}} (24B - 1) \big)$, and $\delta_{A'}$ an inverse of~$-24$ modulo~$A'$. Then we have

\begin{gather*}
  p\Big( \frac{A' n + 1}{24} \Big)
\;\equiv\;
  0
  \;\pmod{\ell^m}	
\end{gather*}
for $n \in \ZZ$ subject to the condition that for all primes~$p \isdiv A'$ we have $\big(\frac{n}{p}\big) = \big(\frac{24 B - 1}{p}\big)$.
\end{maintheorem}

\begin{maintheorem}
\label{mainthm:abundance}
If $\ell \ge 5$ is a prime, $m$ is a positive integer, and~$\epsilon_\ell$ is either~$0$ or~$-\left( \frac{-6}{\ell} \right)$, then for at least one~$\epsilon_Q = \pm 1$ a positive proportion of primes~$Q \equiv -1 \,\pmod{24 \ell}$ have the property that
\begin{gather*}
  p\Big( \frac{Q^2 n + 1}{24} \Big)
\;\equiv\;
  0
  \;\pmod{\ell^m}
\end{gather*}
for all $n \in \ZZ$ with $\big( \frac{n}{\ell} \big) = \epsilon_\ell$ and $\big( \frac{n}{Q} \big) \in \{0, \epsilon_Q \}$.
\end{maintheorem}

Treneer succeeded in extending results of~\cite{ono-2000,ahlgren-ono-2001} to weakly holomorphic modular forms in general~\cite{treneer-2006,treneer-2008}. Her work suggests that also our Theorems~\ref{mainthm:atkin-obrien} and~\ref{mainthm:abundance} should be instances of a framework that applies to all weakly holomorphic modular forms. This is indeed the case. Analogous statements for all weakly holomorphic modular forms are available in Corollaries~\ref{cor:congruences-on-square-classes},~\ref{cor:improving-0-square-class}, and~\ref{cor:improving-nonzero-square-class}. Our strategy to prove them is to leverage results of Deligne-Rapoport~\cite{deligne-rapoport-1973} on compactifications of modular curves over~$\ZZ$ in Theorem~\ref{thm:ell-kernel-representation}. The key observation, stated in Corollary~\ref{cor:congruence-on-arithmetic-progression-hecke-T-eigenvector}, is that congruences of Fourier coefficients of weakly holomorphic modular forms correspond to vectors in specific induced representations of~$\SL{2}(\ZZ)$ or~$\Mp{1}(\ZZ)$, and that by Theorem~\ref{thm:ell-kernel-representation} the submodule associated with congruences modulo~$\ell$ is a subrepresentation. The action of the Cartan subgroups of~$\SL{2}$ and~$\Mp{1}$ on these subrepresentations gives rise to Theorem~\ref{mainthm:atkin-obrien} and Corollary~\ref{cor:congruences-on-square-classes}. Theorem~\ref{mainthm:abundance} and Corollaries~\ref{cor:improving-0-square-class} and~\ref{cor:improving-nonzero-square-class} follow from a more fine-grained analysis of those representations. The explicit description of the extension~$\Mp{1}(\ZZ)$ of~$\SL{2}(\ZZ)$ by~$\{ \pm 1 \}$ due to Kubota~\cite{kubota-1969, weil-1964} plays an important role in our proof.

\paragraph{Acknowledgment}

We are grateful to Olivia Beckwith and Olav Richter for stimulating conversations, and to Claudia Alfes-Neumann, Caihua Luo, Sven Raum, and Olav Richter for helpful remarks on an earlier version of the manuscript.

\section{Abstract spaces of modular forms}

\paragraph{The upper half plane}

The Poincar\'e upper half plane is defined as
\begin{gather*}
  \HS
\;:=\;
  \big\{ \tau \in \CC \,:\, \Im(\tau) > 0 \big\}
\tx{.}
\end{gather*}
It carries an action of~$\GL{2}^+(\RR) := \{ \ga \in \GL{2}(\RR) \,:\, \det(\ga) > 0 \}$ by M\"obius transformations
\begin{gather*}
  \begin{pmatrix} a & b \\ c & d \end{pmatrix} \tau
\;:=\;
  \frac{a \tau + b}{c \tau + d}
\tx{.}
\end{gather*}

\paragraph{The metaplectic group}

The general metaplectic group~$\GMp{1}(\RR)$ is a nontrivial central extension
\begin{gather*}
  1 \lra \mu_2 \lra \GMp{1}(\RR) \lra \GL{2}(\RR) \lra 1
\tx{,}\quad
  \mu_{2} := \{ \pm 1 \}
\tx{.}
\end{gather*}
A common realization of~$\GMp{1}(\RR)$ is given by
\begin{gather*}
  \GMp{1}(\RR)
\,:=\,
  \big\{
  (\ga, \omega) \,:\,
  \ga = \begin{psmatrix} a & b \\ c & d \end{psmatrix} \in \GL{2}(\RR),\,
  \omega :\, \HS \ra \CC,\,
  \omega(\tau)^2 = c \tau + d
  \big\}
\tx{,}
\end{gather*}
in which case multiplication in~$\GMp{1}(\RR)$ is given by
\begin{gather*}
  (\ga, \omega) \cdot (\ga', \omega')
\,:=\,
  (\ga \ga',\, \omega \circ \ga' \cdot \omega')
\tx{.}
\end{gather*}
Throughout, the letter~$\ga$ may denote elements of both~$\GL{2}(\RR)$ and~$\GMp{1}(\RR)$. We let~$\GMp{1}(\QQ)$ and~$\Mp{1}(\ZZ)$ be the preimages of~$\GL{2}(\QQ)$ and~$\SL{2}(\ZZ)$ under the projection~$\GMp{1}(\RR) \ra \GL{2}(\RR)$.

\paragraph{Congruence subgroups}

A subgroup~$\Ga \subseteq \Mp{1}(\ZZ)$ is called a congruence subgroup if its projection to~$\SL{2}(\ZZ)$ is a congruence subgroup. We fix notation for the special cases
\begin{align*}
  \Ga_0(N)
\;&:=\;
  \big\{
  \begin{psmatrix} a & b \\ c & d \end{psmatrix} \in \SL{2}(\ZZ)
  \,:\,
  c \equiv 0 \,\pmod{N}
  \big\}
\tx{,}
\\
  \Ga(N)
\;&:=\;
  \big\{
  \begin{psmatrix} a & b \\ c & d \end{psmatrix} \in \SL{2}(\ZZ)
  \,:\,
  b,c \equiv 0 \,\pmod{N},\, a,d \equiv 1 \,\pmod{N}
  \big\}
\tx{,}
\end{align*}
and write~$\tdGa_0(N)$ and~$\tdGa(N)$ for their preimages under the projection from~$\Mp{1}(\ZZ)$ to~$\SL{2}(\ZZ)$.

We record for later use that the integral metaplectic group~$\Mp{1}(\ZZ)$ is generated by the elements
\begin{gather*}
  S
\,:=\,
  \big( \begin{psmatrix} 0 & -1 \\ 1 & 0 \end{psmatrix}, \tau \mto \sqrt{\tau} \big)
\tx{,}\quad
  T
\,:=\,
  \big( \begin{psmatrix} 1 & 1 \\ 0 & 1 \end{psmatrix}, \tau \mto 1 \big)
\tx{,}
\end{gather*}
where $\sqrt{\tau}$ is the principal branch of the holomorphic square root on~$\HS$.

\paragraph{Slash actions}

The action~$\GL{2}^+(\RR) \circlearrowright \HS$ in conjunction with the cocycle $(\ga,\om) \mto \om$ yields the slash actions of weight~$k \in \frac{1}{2}\ZZ$ on functions~$f :\, \HS \ra \CC$:
\begin{gather}
\label{eq:slash-action-mp1r}
  f \big|_k\, (\ga,\om)
\,:=\,
  \det(\ga)^{\frac{k}{2}} \omega^{-2k} \cdot f \circ \ga
\tx{,}\quad
  (\ga,\om) 
\in
  \GMp{1}^+(\RR)
\tx{.}
\end{gather}

\paragraph{Modular forms}

Fix $k \in \frac{1}{2}\ZZ$, a finite index subgroup $\Ga \subseteq \Mp{1}(\ZZ)$, and a character~$\chi : \Ga \ra \CC^\times$ of finite order. A weakly holomorphic modular form of weight~$k$ for the character~$\chi$ on~$\Ga$ is a holomorphic function~$f :\, \HS \ra \CC$ satisfying the following two conditions:
\begin{enumerateroman}
\item For all $\ga \in \Ga$, we have $f \big|_k\, \ga = \chi(\ga) f$.
\item There exists $a \in \RR$ such that for all~$\ga \in \Mp{1}(\ZZ)$ there is~$b \in \RR$ with
\begin{gather*}
  \big| \big( f \big|_k\, \ga \big) (\tau) \big| < b \exp\big( -a \Im(\tau) \big)
  \tx{\ as\ }\tau \ra i \infty
\tx{.}
\end{gather*}
\end{enumerateroman}
We write~$\rmM^!_k(\Ga, \chi)$ for the space of weakly holomorphic modular forms of weight~$k$ for the character~$\chi$ on~$\Ga$. If~$\chi$ is trivial, we suppress it from our notation.

\paragraph{Group algebras}

Write $\CC[\Ga]$ for the group algebra of a group~$\Ga$, and let $\fraku_\ga \in \CC[\Ga]$ for $\ga \in \Ga$ denote its canonical units. Right-modules of $\CC[\Ga]$ are in canonical correspondence to complex right-representations of~$\Ga$. Given a right module~$V$ of~$\CC[\Ga]$, we write $\ker_{\Ga}(V)$ for the subgroup of~$\Ga$ that acts trivially on~$V$.

Let~$\Ga \subseteq \Mp{1}(\ZZ)$ be a finite index, normal subgroup. The spaces of weakly holomorphic modular forms for~$\Ga$ naturally carry the structure of a right-representation for~$\Mp{1}(\ZZ)$ via the weight-$k$ slash action.

\subsection{Abstract spaces of weakly holomorphic modular forms}

Let $V$ be a finite dimensional right-module for $\CC[\Mp{1}(\ZZ)]$ such that $\ker_{\Mp{1}(\ZZ)}(V) \subseteq \Mp{1}(\ZZ)$ has finite index. An abstract space of weakly holomorphic modular forms is a pair of such a module~$V$ and a homomorphism of~$\CC[\Mp{1}(\ZZ)]$-modules $\phi : V \ra \rmM^!_k(\ker_{\Mp{1}(\ZZ)}(V))$ for some~$k \in \frac{1}{2} \ZZ$. We say that the abstract space of weakly holomorphic modular forms~$(V,\phi)$ is realized in weight~$k$.

\paragraph{Induction of modular forms}

Given a modular form~$f \in \rmM^!_k(\Mp{1}(\ZZ), \chi)$ for a character~$\chi$, the identity map
\begin{gather*}
  \CC f \lra \rmM^!_k(\ker(\chi))
\tx{,}\;
  f \lmto f
\end{gather*}
yields an abstract space of modular forms. More generally, given $f \in \rmM^!_k(\Ga,\chi)$ for a finite index subgroup~$\Ga \subseteq \Mp{1}(\ZZ)$ and a character~$\chi$ of~$\Ga$, we obtain an abstract space of weakly holomorphic modular forms
\begin{gather}
\label{eq:def:induction}
  \Ind_\Ga\,f
\;:=\;
  \Big(
  \CC f \otimes_{\CC[\Ga]} \CC\big[\Mp{1}(\ZZ) \big],\,
  f \otimes \fraku_\ga \mto f \big|_k\,\ga
  \Big)
\tx{,}
\end{gather}
where the action of~$\CC[\Ga]$ on~$\CC f$ is given by the weight~$k$ slash action.

\paragraph{Fourier expansions}

Throughout this paper, we write $e(z)$ for $\exp(2 \pi i\, z)$, $z \in \CC$. Given a finite index subgroup~$\Ga \subseteq \Mp{1}(\ZZ)$ and a character~$\chi$ of~$\Ga$, there is a positive integer~$N$ such that $T^N \in \ker(\chi)$. In particular, a weakly holomorphic modular form~$f$ for the character~$\chi$ has a Fourier series expansion of the form
\begin{gather}
\label{eq:fourier-expansion}
  f(\tau)
\;=\;
  \sum_{n \in \frac{1}{N} \ZZ}
  c(f;\,n) e(n \tau)
\tx{.}
\end{gather}
For convenience, we set~$c(f;\,n) = 0$ if~$n \in \QQ \setminus \frac{1}{N}\ZZ$.

Given any ring~$R$, we write
\begin{gather}
  \rmFE(R)
\;:=\;
  R\big\llbrkt q^{\frac{1}{\infty}} \big\rrbrkt \big[q^{-1}\big]
\end{gather}
for the ring of Puiseux series with coefficients in~$R$. The Fourier expansion of weakly holomorphic modular forms yields a map
\begin{gather*}
  \rmfe :\,
  \rmM^!_k(\Ga, \chi)
\lra
  \rmFE(\CC)
\tx{,}\quad
  f
\lmto
  \rmfe(f)
:=
  \sum_{n \in \QQ} c(f;\,n) q^n
\tx{.}
\end{gather*}
\subsection{\texpdf{$\ell$}{l}-kernels}

We obtain a Fourier expansion map
\begin{gather*}
  \rmfe \circ \phi :\, V \lra \rmFE(\CC)
\end{gather*}
for any abstract space~$(V,\phi)$ of weakly holomorphic modular forms. Given any number field~$K$ with maximal order~$\cO_K$, write $\cO_{K,\ell} \subseteq K$ for the localization of~$\cO_K$ at an ideal~$\ell \subseteq \cO_K$. Fixing an embedding~$K \hra \CC$, we obtain an embedding~$\cO_{K,\ell} \hra \CC$. The $\ell$-kernel of~$(V,\phi)$ is defined as the preimage
\begin{gather}
  \ker_\ell\big( (V,\phi) \big)
\;:=\;
  \big( \rmfe \circ \phi \big)^{-1}\big( \ell \rmFE(\cO_{K,\ell}) \big)
\tx{.}
\end{gather}
Observe that we suppress~$K$ from our notation, but it is implicitly given by~$\ell$. We allow ourselves to identify a rational integer~$\ell$ with the ideal in~$\ZZ$ that it generates.

A priori, the $\ell$-kernel of an abstract space of modular forms is merely an~$\cO_{K,\ell}$-module. The key theorem of this paper is the next one, which guarantees that it is a module for a specific group algebra.
\begin{theorem}
\label{thm:ell-kernel-representation}
Let $K \subseteq \CC$ be a number field with fixed complex embedding, and let~$\ell$ be an ideal in~$\cO_K$. Let~$(V,\phi)$ be an abstract space of weakly holomorphic modular forms realized in weight~$k \in \frac{1}{2}\ZZ$. Assume that~$\ker_{\Mp{1}(\ZZ)}(V)$ is a congruence subgroup of level~$N$. Let $N_\ell$ be the smallest positive integer such that $\gcd(\ell, N \slash N_\ell) = 1$ if $k \in \ZZ$ and $\gcd(\ell, \lcm(4, N) \slash N_\ell) = 1$ if $k \in \frac{1}{2} + \ZZ$. Then the $\ell$-kernel of~$(V,\phi)$ is a right-module for~$\cO_{K,\ell}[\tdGa_0(N_\ell)]$.
\end{theorem}
\begin{remark}
Congruences of half-integral weight modular forms were also investigated by Jochnowitz in unpublished work~\cite{jochnowitz-2004-preprint}. The special case that~$N$ and~$\ell$ are co-prime can also be inferred from her results.
\end{remark}
\begin{proof}
Fix an element~$v \in \ker_\ell((V,\phi))$ and write~$f = \phi(v)$. We have to show that $\phi(v \ga) = f |_k\,\ga$ has Fourier coefficients in $\ell \cO_{K,\ell}$ for every~$\ga \in \tdGa_0(N_\ell)$. We let~$\Delta$ be the Ramanujan $\Delta$-function in $\rmM_{12}(\Mp{1}(\ZZ))$, and observe that~$f \Delta |_k\,\ga$ has Fourier coefficients in~$\ell \cO_{K,\ell}$ if and only if~$f |_k\,\ga$ does. In particular, we can assume that $f$ is a modular form after replacing $\phi$ with~$v \mto \Delta^h \phi(v)$ for sufficiently large~$h \in \ZZ$.

Consider the case $k \in \frac{1}{2} + \ZZ$. Let $\Theta$ be the $\CC$-vector space spanned by the theta series
\begin{gather*}
  \theta_0(\tau)
\;:=\;
  \sum_{n \in \ZZ} e(n^2 \tau)
\tx{,}\quad
  \theta_1(\tau)
\;:=\;
  \sum_{n \in \frac{1}{2} + \ZZ} e(n^2 \tau)
\tx{.}
\end{gather*}
Observe that~$\Theta$ is a representation for~$\Mp{1}(\ZZ)$ under the slash action of weight~$\frac{1}{2}$, which is isomorphic to the dual of the Weil representation associated with the qua\-dra\-tic form~$n \mto n^2$. In particular, we can and will view~$\Theta$ as an abstract space of modular forms. Then $\ker_{\Mp{1}(\ZZ)}(\Theta)$ has level~$4$. We reduce ourselves to the case of integral weight~$k$ at the expense of replacing~$N$ by~$\lcm(4,N)$ when replacing~$(V,\phi)$ with the tensor product $(V,\phi) \otimes \Theta$. Indeed, $f|_k\, \ga$ has Fourier coefficients in~$\ell \cO_{K,\ell}$ if and only if both $f \theta_0 |_{k+\frac{1}{2}}\,\ga$ and $f \theta_1 |_{k+\frac{1}{2}}\,\ga$ do.

In the remainder of the proof, we can and will assume that~$k \in \ZZ$. Since by our assumptions~$\ker_{\Mp{1}(\ZZ)}((V,\phi))$ has level~$N$, $f$ is a modular form for $\Ga(N) \subseteq \SL{2}(\ZZ)$. It therefore falls under Definition~VII.3.6 of modular forms in~\cite{deligne-rapoport-1973} if we suitably enlarge~$K$ by~$N$\thdash\ roots of unity. The transition between the two notions is outlined in Construction~VII.4.6 and~VII.4.7 of~\cite{deligne-rapoport-1973}. In particular, we can apply Corollaire~VII.3.12 to compare $\pi$-adic valuations of the Fourier expansion of~$f$ and~$f |_k\,\ga$ for every place~$\pi$ of~$K$ lying above~$\ell$. Using the notation by Deligne-Rapoport~\cite{deligne-rapoport-1973}, the condition on the image of their $g \in \SL{}(2, \ZZ \slash n)$ in the group~$\SL{}(2, \ZZ \slash p^m)$ translates into our condition that~$\gcd(\ell, N \slash N_\ell) = 1$.
\end{proof}

\subsection{Hecke operators}

Given a positive integer~$M$, we let
\begin{gather*}
  \GL{2}^{(M)}(\ZZ)
:=
  \big\{
  \ga = \begin{psmatrix} a & b \\ c & d \end{psmatrix} \in \Mat{2}(\ZZ) \,:\,
  \det(\ga) = M
  \big\}
\tx{,}
\end{gather*}
and correspondingly write $\GMp{1}^{(M)}(\ZZ)$ for its preimage in~$\GMp{1}(\RR)$ under the projection to~$\GL{2}(\RR)$. Let $\CC[ \GMp{1}^{(M)}(\ZZ) ]$ be the $\CC[\Mp{1}(\ZZ)]$-bi-module with $\CC$-basis~$\fraku_\ga$, $\ga \in \GMp{1}^{(M)}(\ZZ)$. Given an abstract space of weakly holomorphic modular forms~$(V,\phi)$, we define the application of the $M$\thdash\ Hecke operator by
\begin{gather}
\label{eq:def:hecke-operator-abstract-space-of-modular-forms}
  \rmT_M\,(V,\phi)
\;:=\;
  \Big(
  V \otimes_{\CC[\Mp{1}(\ZZ)]} \CC\big[ \GMp{1}^{(M)}(\ZZ) \big],\;
  v \otimes \fraku_\ga \mto \phi(v) \big|_k\,\ga
  \Big)
\tx{.}
\end{gather}
It is straightforward to verify that~$\rmT_M\,(V,\phi)$ is an abstract space of weakly holomorphic modular forms. By slight abuse of notation, we write $\rmT_M\,(V,\phi) = (\rmT_M\,V,\rmT_M\,\phi)$.

\subsection{Congruences of modular forms on arithmetic progressions}

Throughout, we will identify upper triangular matrices~$\ga \in \GL{2}^+(\RR)$ with~$(\ga, \tau \mto \sqrt{\det(\ga)})$. Fix an abstract space of weakly holomorphic modular forms~$(V,\phi)$ and an eigenvector~$v \in V$ for~$T \in \Mp{1}(\ZZ)$. Choose~$\alpha \in \QQ$ such that~$v T = e(\alpha) v$. Given a positive integer~$M$ and an integer~$b$, we let
\begin{gather}
\label{eq:def:hecke-T-eigenvectors}
  v_T(M,b;\alpha)
\;:=\;
  v \otimes
  \sum_{h \,\pmod{M}}
  e\big( - \tfrac{(b+\alpha)h}{M} \big)\,
  \fraku_{\begin{psmatrix} 1 & h \\ 0 & M \end{psmatrix}}
  \,\in\,
  \rmT_M\,V
\tx{.}
\end{gather}
A brief verification shows that the sum over~$h \,\pmod{M}$ is well-defined and that we have
\begin{gather}
  v_T(M,b;\alpha) T
\;=\;
  e\big(\tfrac{b+\alpha}{M}\big) v_T(M,b;\alpha)
\tx{.}
\end{gather}

The image of~$v_T(M,b;\alpha)$ under~$\rmT_M\,\phi$ captures the Fourier coefficients of~$\phi(v)$ on the arithmetic progression~$M \ZZ + b + \alpha$. Specifically, we have the next proposition:
\begin{proposition}
\label{prop:hecke-T-eigenvectors-fourier-expansion}	
Given a positive integer~$M$, an integer~$b$, an abstract space of weakly holomorphic modular forms~$(V,\phi)$, and a $T$-eigenvector~$v \in V$ with eigenvalue~$e(\alpha)$, $\alpha \in \QQ$, we have
\begin{gather}
\label{eq:prop:hecke-T-eigenvectors-fourier-expansion}
  (\rmT_M\phi) \big( v_T(M,b;\alpha) \big)
\;=\;
  M^{1-\frac{k}{2}}\,
  \sum_{\substack{n \in \alpha + \ZZ\\n-\alpha \equiv b \,\pmod{M}}}
  c(f;\,n)
  e\big( \tfrac{n}{M}\, \tau\big)
\tx{.}
\end{gather}
\end{proposition}
\begin{proof}
The definition of the slash action in~\eqref{eq:slash-action-mp1r} yields
\begin{gather*}
  e(n \tau) \big|_k\, \begin{psmatrix} 1 & h \\ 0 & M \end{psmatrix}
\;=\;
  M^{-\frac{k}{2}}\,
  e\big( \tfrac{h}{M} n \big)\,
  e\big( \tfrac{n}{M}\, \tau\big)
\tx{.}
\end{gather*}
Let $f = \phi(v)$, and observe that $v T = e(\alpha) v$ implies that $c(f;\,n)$ is supported on~$n \in \alpha + \ZZ$. When inserting the definition of~$v_T(M,b;\alpha)$ and the Fourier expansion of~$f$, we find that
\begin{align*}
  (\rmT_M\phi) \big( v_T(M,b;\alpha) \big)
\;&=\;
  \sum_{h \,\pmod{M}}
  e\big( - \tfrac{(b+\alpha)h}{M} \big)\,
  f \big|_k\, \big( \begin{psmatrix} 1 & h \\ 0 & M \end{psmatrix},\, \tau \mto \sqrt{M} \big)
\\
&=\;
  M^{-\frac{k}{2}}\,
  \sum_{n \in \alpha + \ZZ}
  c(f;\,n)
  e\big( \tfrac{n}{M}\, \tau\big)
  \sum_{h \,\pmod{M}}
  e\big( \tfrac{h(n - \alpha - b)}{M} \big)
\\
&=\;
  M^{1-\frac{k}{2}}\,
  \sum_{\substack{n \in \alpha + \ZZ\\n-\alpha \equiv b \,\pmod{M}}}
  c(f;\,n)
  e\big( \tfrac{n}{M}\, \tau\big)
\tx{.}
\end{align*}
\end{proof}

\begin{corollary}
\label{cor:congruence-on-arithmetic-progression-hecke-T-eigenvector}
Let~$(V,\phi)$ be an abstract space of weakly holomorphic modular forms, $K \subseteq \CC$ a number field and~$\ell$ an ideal in~$\cO_K$. Fix a~$T$-eigenvector~$v \in V$ with eigenvalue~$e(\alpha)$, $\alpha \in \QQ$, and write $f = \phi(v)$. Given a positive integer~$M$ and an integer~$b$, the congruences
\begin{gather*}
\label{eq:cor:congruence-on-arithmetic-progression-hecke-T-eigenvector:arithmetic-progression}
  \forall n \in \ZZ :\,
  c(f;\,M n + b + \alpha) \equiv 0 \;\pmod{\ell}
\end{gather*}
are equivalent to
\begin{gather*}
\label{eq:cor:congruence-on-arithmetic-progression-hecke-T-eigenvector:hecke}
  v_T(M,b;\alpha)
\in
  \ker_\ell\big( \rmT_M (V,\phi) \big)
\tx{.}
\end{gather*}
\end{corollary}

\section{The interplay of congruences}

In this section, we derive from Theorem~\ref{thm:ell-kernel-representation} and Corollary~\ref{cor:congruence-on-arithmetic-progression-hecke-T-eigenvector} implications for general weakly holomorphic modular forms, tailor-made to establish Theorems~\ref{mainthm:atkin-obrien} and~\ref{mainthm:abundance} in Section~\ref{sec:proof-of-main-theorems}. In this section~$p$ does \emph{not} denote the partition function, but a prime.

We will employ Kubota's description of~$\GMp{1}(\QQ)$ in~\cite{kubota-1969}.
In order to distinguish it from our previous description of~$\GMp{1}(\QQ)$, we denote it by~$\GMp{1}'(\QQ)$. As a set, we have~$\GMp{1}'(\QQ) = \GL{2}(\QQ) \times \mu_2$, $\mu_2 = \{ \pm 1 \}$, which yields a section (of sets)~$\GL{2}(\QQ) \hra \GMp{1}(\QQ)$. In this situation, multiplication in~$\Mp{1}(\QQ)$ is given by
\begin{gather*}
  (\ga,\omega) \cdot (\ga', \omega')
\;=\;
  \big( \ga \ga', \omega + \omega' + \psi_b(\ga, \ga') \big)
\tx{,}
\end{gather*}
where $\psi_b$ arises from Kubota's cocycle~$b$ defined in~(21) on page~22 of~\cite{kubota-1969}. Specifically, $\psi_b$ splits into a product of local contributions~$\prod \psi_{b,p}$, where~$p$ runs through all prime numbers. Theorem~2 of~\cite{kubota-1969} implies that~$\psi_b$ is trivial on~$\Ga(4)$. In particular, we obtain a section of groups~$\Ga(4) \hra \Mp{1}'(\ZZ)$. Finally, we record two further consequences of Kubota's Theorem~2: First, any other section~$\Ga(4) \hra \Mp{1}'(\ZZ)$ differs from Kubota's by a homomorphism~$\Ga(4) \ra \mu_2$. Second, $N$ is even if there is a section~$\Ga(N) \hra \Mp{1}'(\ZZ)$.

\begin{proposition}
\label{prop:hecke-T-eigenvector-cartan-orbit}
Let~$M$ and~$N$ be co-prime positive integers, let $(V,\phi)$ be an abstract space of modular forms, and assume that\/~$\ker_{\Mp{1}(\ZZ)}(V)$ is a congruence subgroup of level~$N$. Let $W \subseteq \rmT_M\,V$ be a subrepresentation for~$\tdGa_0(M) \cap \tdGa(N)$. Given an eigenvector~$v \in V$ for~$T \in \Mp{1}(\ZZ)$ with eigenvalue~$e(\alpha)$,~$\alpha \in \QQ$, assume that $v_T(M, b; \alpha) \in W$ for some integer~$b \in \ZZ$. Then for every~$u \in \ZZ$ co-prime to~$M$ and congruent to~$1 \pmod{N}$, we have $v_T(M, u^2(b + \alpha) - \alpha; \alpha) \in W$.
\end{proposition}

\begin{corollary}
\label{cor:congruences-on-square-classes}
Fix a number field $K \subseteq \CC$ and an ideal~$\ell \subseteq \cO_K$. Let $f \in \rmM_k(\tdGa(N))$ be a modular form with Fourier coefficients~$c(f;\,n) \in \cO_{K,\ell}$ supported on~$n \in \alpha + \ZZ$ for some~$\alpha \in \frac{1}{N}\ZZ$. Assume that~$f$ satisfies the congruences
\begin{gather*}
  \forall n \in \ZZ \,:\, c(f;\,M n + b + \alpha) \equiv 0 \;\pmod{\ell}
\end{gather*}
for some positive integer~$M$ that is co-prime to~$N$ and some~$b \in \ZZ$. Then
\begin{gather*}
  \forall n \in \ZZ \,:\, c(f;\,M n + b' + \alpha) \equiv 0 \;\pmod{\ell}
\end{gather*}
for all $b' \in \ZZ$ with $b' + \alpha \equiv u^2 (b + \alpha)$ for some~$u \in \ZZ$ that is co-prime to~$M$.
\end{corollary}
\begin{proof}
Consider the abstract space of modular forms~$\Ind_{\Mp{1}(\ZZ)}\,f$ and combine Corollary~\ref{cor:congruence-on-arithmetic-progression-hecke-T-eigenvector} with Proposition~\ref{prop:hecke-T-eigenvector-cartan-orbit}.
\end{proof}

\begin{proof}[{Proof of Proposition~\ref{prop:hecke-T-eigenvector-cartan-orbit}}]
If $N$ is even, assume that~$8 \isdiv N$ by replacing $N$ if necessary. Fix~$u$ as in the statement and let~$\ov{u} \in \ZZ$ denote its inverse modulo~$M N$. There is a transformation~$\ga \in \tdGa_0(M) \cap \tdGa(N)$ whose image~$\ov{\ga}$ in~$\SL{2}(\ZZ)$ is congruent to $\begin{psmatrix} u & 0 \\ 0 & \ov{u} \end{psmatrix}$ modulo~$M N$. For any such~$\ga$ and~$h \in \ZZ$, we have
\begin{gather*}
  \ov{\ga}^{-1}\, \begin{pmatrix} 1 & h \\ 0 & M \end{pmatrix} \,\ov{\ga}
\;\equiv\;
  \begin{pmatrix} 1 & \ov{u}^2 h \\ 0 & M \end{pmatrix}
  \;\pmod{M N}
\tx{.}
\end{gather*}
In particular, we have
\begin{gather*}
   \begin{pmatrix} 1 & h \\ 0 & M \end{pmatrix}
   \,\ga\,
   \begin{pmatrix} 1 & \ov{u}^2 h \\ 0 & M \end{pmatrix}^{-1}
\;\in\;
   \tdGa(N)
\tx{.}
\end{gather*}

The corresponding contribution to the center of~$\Mp{1}(\ZZ)$ is captured by
\begin{gather*}
  \omega_{h,u}
\,:=\,
  \begin{pmatrix} 1 & h \\ 0 & M \end{pmatrix}
  \,\ga\,
  \begin{pmatrix} 1 & \ov{u}^2 h \\ 0 & M \end{pmatrix}^{-1}
  \,\in\,
  \tdGa(N) \slash \big( \tdGa(N) \cap \ker_{\Mp{1}(\ZZ)}(V) \big)
\tx{.}
\end{gather*}
Without loss of generality, we can and will assume that the center of~$\Mp{1}(\ZZ)$ acts by scalars on~$V$. Write~$\wtd\omega_{h,u} \in \{\pm 1\}$ for the action of~$\omega_{h,u}$. We then have
\begin{align*}
  v_T(M,b;\alpha) \ga
&{}=
  v \otimes
  \sum_{h \,\pmod{M}}
  e\big(-\tfrac{h(b + \alpha)}{M} \big)\,
  \fraku_{\begin{psmatrix} 1 & h \\ 0 & M \end{psmatrix}} \ga
\\
&{}=
  v
  \otimes
  \sum_{h \,\pmod{M}}
  e\big(-\tfrac{h(b + \alpha)}{M} \big)\,
  \wtd\omega_{h,u}
  \fraku_{\begin{psmatrix} 1 & \ov{u}^2 h \\ 0 & M \end{psmatrix}}
\\
&{}=
  v
  \otimes
  \sum_{h \,\pmod{M}}
  e\big(-\tfrac{h u^2 (b + \alpha)}{M} \big)\,
  \wtd\omega_{h u^2,u}
  \fraku_{\begin{psmatrix} 1 & h \\ 0 & M \end{psmatrix}}
\tx{.}
\end{align*}
If we show that~$\omega_{h u^2, u}$ and hence~$\wtd\omega_{h u^2, u}$ is independent of~$h$, the very right hand side equals~$\pm v_T(M, u^2 (b + \alpha) - \alpha; \alpha)$ on the nose, and the proposition follows.

We will perform the computation of~$\omega_{h,u}$ in~$\Mp{1}'(\ZZ)$ as opposed to~$\Mp{1}(\ZZ)$. Notice that
\begin{gather*}
  \begin{psmatrix} 1 & h \\ 0 & M \end{psmatrix}
=
  \begin{psmatrix} 1 & 0 \\ 0 & M \end{psmatrix}
  T^h
\quad\tx{and}\quad
  T^h
=
  \big( \begin{psmatrix} 1 & h \\ 0 & 1 \end{psmatrix}, 0 \big)
  \in
  \Mp{1}'(\ZZ)
\tx{.}
\end{gather*}
This fixes the contribution of both~$\begin{psmatrix} 1 & h \\ 0 & M \end{psmatrix} \in \GMp{1}(\QQ)$ and~$\begin{psmatrix} 1 & \ov{u}^2 h \\ 0 & M \end{psmatrix} \in \GMp{1}(\QQ)$ to $\mu_2 \subset \GMp{1}'(\QQ)$. In particular, they are independent of~$h$.

We can now split the computation of~$\omega_{h,u}$ into local considerations over the~$p$-adic field~$\QQ_p$, and compare
\begin{gather}
\label{eq:prop:hecke-T-eigenvector-cartan-orbit:cocycles}
  \psi_{b,p}\big(
  \begin{psmatrix} 1 & h \\ 0 & M \end{psmatrix},\, \ov\ga
  \big)
\quad\tx{and}\quad
  \psi_{b,p}\big(
  \ov\ga, \begin{psmatrix} 1 & \ov{u}^2 h \\ 0 & M \end{psmatrix}
  \big)
\end{gather}
for each prime~$p$. If~$p \nisdiv M N$, then both are~$1$ by Theorem~2 of~\cite{kubota-1969}. We can therefore assume that~$p \isdiv M N$. Denote the bottom left entry of~$\ov{\ga}$ by~$c$. By adjusting the choice of~$\ov{\ga}$, we can and will assume that~$c \ne 0$. Inserting Kubota's formula (pp.~16, 17, 19, and~21 of~\cite{kubota-1969}), we find that the cocycle values in~\eqref{eq:prop:hecke-T-eigenvector-cartan-orbit:cocycles} are $(1,c)_p (-c,c)_p (1,1)_p = 1$ and~$(c \slash M, 1)_p (- M \slash c, c \slash M)_p = 1$, where $(\,\cdot,\cdot\,)_p$ is the Hilbert residue symbol. Both are independent of~$h$, as desired.
\end{proof}

\begin{proposition}
\label{prop:hecke-T-eigenvector-p4-case}
Let~$M$ and~$N$ be co-prime positive integers, let $(V,\phi)$ be an abstract space of modular forms, and assume that $\ker_{\Mp{1}(\ZZ)}(V)$ is a congruence subgroup of level~$N$. Let $p \isdiv M$ be a prime and $M = M_p M_p^\#$ a factorization of~$M$ satisfying~$M_p \isdiv p^4$ and~$\gcd(p,M_p^\#) = 1$. Let $W \subseteq \rmT_M\,V$ be a subrepresentation for~$\tdGa(M_p^\# N)$. Given an eigenvector~$v \in V$ for~$T \in \Mp{1}(\ZZ)$ with eigenvalue~$e(\alpha)$, $\alpha \in \QQ$, that generates a character for~$\tdGa_0(N)$, assume that $v_T(M, b; \alpha) \in W$ for some integer~$b \in M_p \ZZ$. Then we have $v_T(M, b + u M \slash p; \alpha) \in W$ for some~$u \in \ZZ$ co-prime to~$p$.
\end{proposition}

\begin{corollary}
\label{cor:improving-0-square-class}
Fix a number field $K \subseteq \CC$ and an ideal~$\ell \subseteq \cO_K$. Let $f \in \rmM_k(\tdGa_0(N),\chi)$ be a modular form for a character~$\chi$ whose kernel has level~$N$ with Fourier coefficients~$c(f;\,n) \in \cO_{K,\ell}$ supported on~$n \in \alpha + \ZZ$ for some~$\alpha \in \frac{1}{N}\ZZ$. Assume that~$f$ satisfies the congruences
\begin{gather*}
  \forall n \in \ZZ \,:\, c(f;\,M n + b+ \alpha) \equiv 0 \;\pmod{\ell}
\end{gather*}
for some positive integer~$M$ that is co-prime to~$N$ and some~$b \in M_p \ZZ$, where $M = M_p M_p^\#$ is a factorization of~$M$ satisfying~$M_p \isdiv p^4$ and~$\gcd(p,M_p^\#) = 1$. Then
\begin{gather*}
  \forall n \in \ZZ \,:\, c(f;\,M n + b + u M \slash p + \alpha) \equiv 0 \;\pmod{\ell}
\end{gather*}
for some~$u \in \ZZ$ co-prime to~$p$.
\end{corollary}

\begin{corollary}
\label{cor:improving-nonzero-square-class}
Let~$K$, $\ell$, $f$, $N$, $\alpha$, $M$, $M_p$, $M_p^\#$ be as in Corollary~\ref{cor:improving-0-square-class}. Fix~$b \in M_p \ZZ$. Then there is $u \in \ZZ$ co-prime to~$p$ with the following property: If~$f$ satisfies the congruence
\begin{gather*}
  \forall n \in \ZZ \,:\, c(f;\,M n + b + u M \slash p + \alpha) \equiv 0 \;\pmod{\ell}
\tx{,}
\end{gather*}
then it also satisfies the congruences
\begin{gather*}
  \forall n \in \ZZ \,:\, c(f;\,M n + b + \alpha) \equiv 0 \;\pmod{\ell}
\tx{.}
\end{gather*}
\end{corollary}
\begin{proof}[{Proof of Corollaries~\ref{cor:improving-0-square-class} and~\ref{cor:improving-nonzero-square-class}}]
Consider $\Ind_{\Mp{1}(\ZZ)}\,f$,the abstract space of modular forms associated with~$f$ via~\eqref{eq:def:induction}, and combine Corollary~\ref{cor:congruence-on-arithmetic-progression-hecke-T-eigenvector} with Proposition~\ref{prop:hecke-T-eigenvector-p4-case}.
\end{proof}

\begin{proof}[{Proof of Proposition~\ref{prop:hecke-T-eigenvector-p4-case}}]
If $N$ is even, assume that~$8 \isdiv N$ by replacing $N$ if necessary. We use a similar reasoning as in the proof of Proposition~\ref{prop:hecke-T-eigenvector-cartan-orbit} employing a transformation~$\ga \in \tdGa(M_p^\# N)$ whose image~$\ov{\ga}$ in~$\SL{2}(\ZZ)$ is congruent modulo~$M_p$ to~$\begin{psmatrix} 1 & 0 \\ 1 & 1 \end{psmatrix}$.

We split up~$h = h_p + h_p^\#$, yielding the following expression for~$v_T(M,b;\alpha)$:
\begin{gather*}
  v_T(M,b;\alpha)
\;=\;
  v \otimes
  \sum_{\substack{h_p \,\pmod{M_p}\\ h_p^\# \,\pmod{M_p^\#}}}
  e\big( - \tfrac{h_p \alpha}{M_p} - \tfrac{h_p^\# (b + \alpha)}{M_p^\#} \big)
  \fraku_{\begin{psmatrix} 1 & h \\ 0 & M \end{psmatrix}}
\tx{,}
\end{gather*}
where used the assumption that $M_p$ divides~$b$ and we assume that every $h_p$ is divisible by~$M_p^\# N$ and every $h_p^\#$ is divisible by~$M_p N$. We let~$1_p \in M_p^\# N \ZZ$ be congruent to~$1 \,\pmod{M_p}$. Further, we factorize $1 + h_p = p^n h'_p$ where~$\gcd(p,h'_p) = 1$ and let $\ov{h}'_p$ be an inverse of~$h'_p$ modulo~$M N \slash p^n$. It is helpful to record that~$1_p M \equiv 0 \,\pmod{M N}$ and~$h'_p \equiv 1 \slash p^n \,\pmod{M_p^\# N}$. We have
\begin{gather*}
  \ga_h
\;:=\;
  \begin{pmatrix} 1 & h \\ 0 & M \end{pmatrix}
  \,\ga\,
  \begin{pmatrix} p^n & \ov{h}'_p h \\ 0 & M \slash p^n \end{pmatrix}^{-1}
  \,\in\,
  \tdGa_0(N)
\tx{,}
\end{gather*}
since
\begin{gather*}
  \begin{pmatrix} 1 & h \\ 0 & M \end{pmatrix}
  \ov{\ga}
\;\equiv\;
  \begin{pmatrix} 1 & h \\ 0 & M \end{pmatrix}
  \begin{pmatrix} 1 & 0 \\ 1_p & 1 \end{pmatrix}
\;\equiv\;
  \begin{pmatrix} 1 + h_p & h \\ 0 & M \end{pmatrix}
\equiv\;
  \begin{pmatrix} h'_p & 0 \\ 0 & \ov{h}'_p \end{pmatrix}
  \begin{pmatrix} p^n & \ov{h}'_p h \\ 0 & h'_p M \end{pmatrix}
  \;\pmod{M N}
\tx{.}
\end{gather*}
This yields the expression
\begin{gather}
\label{eq:prop:hecke-T-eigenvector-p4-case:pn-contribution}
  v_T(M,b;\alpha) \ga
\;=\;
  \sum_{\substack{p^n \isdiv M_p \\ h'_p \,\pmod{M_p \slash p^n}^\times \\ h_p^\# \,\pmod{M_p^\#}}}
  e\big( - \tfrac{h_p \alpha}{M_p} - \tfrac{h_p^\# (b + \alpha)}{M_p^\#} \big)
  v \ga_h\, \otimes
  \fraku_{\begin{psmatrix} p^n & \ov{h}'_p h \\ 0 & M \slash p^n \end{psmatrix}}
\tx{.}
\end{gather}
By the assumptions~$\ga_h$ acts on~$v$ by a scalar~$\wtd\ga_h$.

Without loss of generality, we can and will assume that~$V$ is irreducible. We next inspect the decomposition of~\eqref{eq:prop:hecke-T-eigenvector-p4-case:pn-contribution} into eigenvectors of the~$M_p^\# N$\thdash\ power~$T'$ of~$T \in \Mp{1}(\ZZ)$. We have~$T' \in \tdGa(M_p^\# N)$. As a consequence, every such eigenvector is contained in~$W$, since the center of~$\Mp{1}(\ZZ)$ acts by scalars. If~$M_p \isdiv p^{2n}$, then
\begin{gather}
\label{eq:prop:hecke-T-eigenvector-p4-case:T'-action}
  \begin{pmatrix} p^n & \ov{h}'_p h \\ 0 & M \slash p^n \end{pmatrix}
  T'
\;=\;
  \begin{pmatrix} p^n & p^n M_p^\# N + \ov{h}'_p h \\ 0 & M \slash p^n \end{pmatrix}
\;=\;
  T^{p^{2n} \slash M_p}
  \begin{pmatrix} p^n & \ov{h}'_p h \\ 0 & M \slash p^n \end{pmatrix}
\tx{.}
\end{gather}
In particular, $T'$ acts trivially on contributions to~\eqref{eq:prop:hecke-T-eigenvector-p4-case:pn-contribution} from~$n \ge 2$ or from $n = 1$ if $M_p \isdiv p^2$. We want to show that this also holds for~$n = 1$ if $M_p \nisdiv p^2$.

We notice that the image~$\ov\ga_h$ of~$\ga_h$ in~$\SL{2}(\ZZ)$ is congruent modulo~$N$ to a diagonal matrix with nonzero entries~$1 \slash p^n$ and~$p^n$. In particular, for fixed~$n = 1$, $\wtd\ga_h$ is constant up to a possible contribution of~$\tdGa(N) \slash (\tdGa(N) \cap \ker_{\Mp{1}(\ZZ)}(V))$. We will determine that contribution as in the proof of Proposition~\ref{prop:hecke-T-eigenvector-cartan-orbit}. To this end, we observe that we have $\ov{h}'_p h \equiv p - \ov{h}'_p \,\pmod{M_p}$. Therefore, we can record that the value of~$h'_p \,\pmod{p}$ remains constant, when~$T'$ acts as on the left hand side of~\eqref{eq:prop:hecke-T-eigenvector-p4-case:T'-action}. It remains to determine the cocycles
\begin{gather}
\label{eq:prop:hecke-T-eigenvector-p4-case:cocycles}
  \psi_{b,q}\Big(
  \begin{psmatrix} 1 & h \\ 0 & M \end{psmatrix},\,
  \ov{\ga}
  \Big)
\quad\tx{and}\quad
  \psi_{b,q}\Big(
  \ov\ga_h,\,
  \begin{psmatrix} p^n & \ov{h}'_p h \\ 0 & M \slash p^n \end{psmatrix}
  \Big)
\end{gather}
for primes~$q \isdiv M N$ provided that~$h'_p \,\pmod{p}$ is constant. Employing Kubota's formula as before, we find that the first cocycle value in~\eqref{eq:prop:hecke-T-eigenvector-p4-case:cocycles} equals~$(1,1)_q (-1,1)_q (1,1)_q = 1$. The second one simplifies to
\begin{gather*}
  \big( 1 \slash h'_p, 1 \slash p \big)_q
  \big( -h'_p \slash p, 1 \slash h'_p p \big)_q
  \big( M, 1 \slash h'_p \big)_q
\tx{.}
\end{gather*}
This equals~$1$ if~$q \ne p$ and it is constant provided that~$h'_p$ is constant modulo~$p$.

Summarizing, we have shown that $T'$ acts trivially on the contributions to~\eqref{eq:prop:hecke-T-eigenvector-p4-case:pn-contribution} from~$n \ne 0$. It does not act trivially on contributions to~\eqref{eq:prop:hecke-T-eigenvector-p4-case:pn-contribution} from~$n = 0$, since all corresponding eigenvectors are linear combinations of~$v_T(M,b';\alpha)$ with $M_p \isdiv b'$. Since the eigenvalue under the $M_p$\thdash\ power of~$T$ did not change, we conclude that $v_T(M,b + u M \slash p;\, \alpha) \in W$ for some~$u$ that is co-prime to~$p$, as stated in the proposition.
\end{proof}

\section{Proof of the main theorems}
\label{sec:proof-of-main-theorems}

The proofs of Theorems~\ref{mainthm:atkin-obrien} and~\ref{mainthm:abundance} rely on the generating function
\begin{gather}
\label{eq:partition-generating-function}
  \eta^{-1}(\tau)
\;=\;
  \sum_{n \in - \frac{1}{24} + \ZZ}
  p\big( n + \tfrac{1}{24} \big) e(n \tau)
\tx{,}
\end{gather}
which is a weakly holomorphic modular form for a specific character~$\chi$ of~$\Mp{1}(\ZZ)$ whose kernel has level~$24$.

\begin{proof}[{Proof of Theorem~\ref{mainthm:atkin-obrien}}]
Combining the assumption with~\eqref{eq:partition-generating-function}, we find that
\begin{gather*}
  c\big( \eta^{-1};\, A n + B - \tfrac{1}{24} \big)
\;\equiv\;
  0
  \;\pmod{\ell^m}
\end{gather*}
for all~$n \in \ZZ$. We apply Corollaries~\ref{cor:congruence-on-arithmetic-progression-hecke-T-eigenvector} and~\ref{cor:congruences-on-square-classes} to find that for all~$u \in \ZZ$ that are co-prime to~$A$ and congruent to~$1$ modulo~$24$, we have further congruences
\begin{gather*}
  c\big( \eta^{-1};\, A n + u^2 (B - \tfrac{1}{24}) \big)
\;\equiv\;
  0
  \;\pmod{\ell^m}
\end{gather*}
for all integers~$n$. Translating back to the partition function and rewriting slightly, we have
\begin{gather*}
  p\Big( \frac{24 A n + u^2 (24 B - 1) + 1}{24} \Big)
\;\equiv\;
  0
  \;\pmod{\ell^m}
\end{gather*}
for all~$n \in \ZZ$.

It remains to describe the set of integers
\begin{gather}
\label{eq:mainthm:atkin-obrien:indices}
  A n + u^2 (24 B - 1)
\tx{,}\quad
  n \in \ZZ,  u \in \ZZ, \gcd(u, A) = 1
\tx{,}
\end{gather}
where we have suppressed the factor~$24$ in front of~$A$ and then the condition~$u \equiv 1 \,\pmod{24}$, which is legitimate, since~$p(n) = 0$ for nonintegral~$n$. To ease notation, set~$\td{B} = 24 B - 1$. We have to show that~\eqref{eq:mainthm:atkin-obrien:indices} runs through all integers of the form~$A' n'$ with $n' (\ZZ \slash A' \ZZ)^{\times\,2} = \td{B} (\ZZ \slash A' \ZZ)^{\times\,2}$, where $A'$ is as in the theorem.

Let~$A''$ be the smallest divisor of~$A$ such that
\begin{multline*}
  \big\{
  A'' n + u^2 \td{B}
  \,:\,
  n \in \ZZ, u \in \ZZ, \gcd(u, A'') = 1
  \big\}
\\
=
  \big\{
  A n + u^2 \td{B}
  \,:\,
  n \in \ZZ, u \in \ZZ, \gcd(u, A) = 1
  \big\}
\tx{.}
\end{multline*}
Clearly, $A' \isdiv A''$. It suffices to show that~$A'' \isdiv A'$. If this is not true, then there is a prime~$Q \isdiv A''$ such that $Q^2 \gcd(A'',\td{B}) \isdiv A''$. Then~$\td{B} + A'' \slash Q$ lies in the same square-class modulo~$A$ as~$\td{B}$. We can therefore replace~$A''$ by~$A'' \slash Q$, which contradicts minimality of~$A''$.
\end{proof}

\begin{proof}[{Proof of Theorem~\ref{mainthm:abundance}}]
We start the proof by employing Theorem~1 of~\cite{ahlgren-ono-2001}. Let~$\delta_\ell$ be an inverse of~$-24 \,\pmod{\ell}$. Fix~$\epsilon_\ell$ as in our Theorem~\ref{mainthm:abundance} and observe that there is at least one~$\beta \in S_\ell$ of~\cite{ahlgren-ono-2001} such that~$\big(\frac{\beta + \delta_\ell}{\ell} \big) = \epsilon_\ell$. Then Theorem~1 of~\cite{ahlgren-ono-2001} states that a positive proportion of primes~$Q \equiv -1 \,\pmod{24 \ell}$ has the property that
\begin{gather*}
  p\Big( \frac{Q^3 (24 \ell n + 1 - 24 \beta) + 1}{24} \Big)
\;\equiv\;
  0
  \;\pmod{\ell^m}
\end{gather*}
for integers~$n$ such that~$24 \ell n + 1 - 24 \beta$ is co-prime to~$Q$. In order to accommodate this condition, we replace~$n$ by~$Q n + n_0$, where $n_0 \in \ZZ$ will be fixed later subject to the condition that~$\gcd(Q, 24 \ell n_0 + 1 - 24 \beta) = 1$. Translating the resulting partition congruence into a congruence for the Fourier coefficients of~$\eta^{-1}$, we have
\begin{gather}
\label{eq:mainthm:abundance:q4-congruence}
  c\Big( \eta^{-1};\, Q^4 \ell n + Q^3 \frac{24 \ell n_0 + 1 - 24 \beta}{24} \Big)
\;\equiv\;
  0
  \;\pmod{\ell^m}
\end{gather}
for all~$n \in \ZZ$.

Let~$n'_0$ be such that~$Q \isdiv 24 \ell n'_0 + 1 - 24 \beta$. We apply Corollary~\ref{cor:improving-nonzero-square-class} with~$M \leadsto Q^4 \ell$, $N \leadsto 24$, $\alpha \leadsto -1 \slash 24$, and~$b \leadsto (Q^3 (24 \ell n'_0 + 1 - 24 \beta) + 1) \slash 24$. We can now fix~$n_0 = n'_0 + u$, where~$u$ is as in Corollary~\ref{cor:improving-nonzero-square-class}. The congruences in~\eqref{eq:mainthm:abundance:q4-congruence} then imply via Corollary~\ref{cor:improving-nonzero-square-class} that
\begin{gather}
\label{eq:mainthm:abundance:q3-congruence}
  c\Big( \eta^{-1};\, Q^3 \ell n + Q^3 \frac{1 - 24 \beta}{24} \Big)
\;\equiv\;
  0
  \;\pmod{\ell^m}
\end{gather}
for all~$n \in \ZZ$.

We next apply Corollary~\ref{cor:improving-0-square-class} to~\eqref{eq:mainthm:abundance:q3-congruence} and obtain a congruence
\begin{gather*}
  c\Big( \eta^{-1};\, Q^3 \ell n + Q^3 \frac{1 - 24 \beta}{24} + 24 Q^2 \ell u' \Big)
\;\equiv\;
  0
  \;\pmod{\ell^m}
\end{gather*}
for all~$n \in \ZZ$ and for some~$u'$ that is co-prime to~$Q$.

To finish the proof, we let~$\epsilon_Q = \big(\frac{24 \ell u'}{Q} \big)$ and apply Corollary~\ref{cor:congruences-on-square-classes} as in the proof of Theorem~\ref{mainthm:atkin-obrien}.
\end{proof}

\renewbibmacro{in:}{}
\renewcommand{\bibfont}{\normalfont\small\raggedright}
\renewcommand{\baselinestretch}{.8}

\Needspace*{4em}
\begin{multicols}{2}
\printbibliography[heading=none]%
\end{multicols}

\Needspace*{3\baselineskip}
\noindent
\rule{\textwidth}{0.15em}

{\noindent\small
Chalmers tekniska högskola och G\"oteborgs Universitet,
Institutionen för Matematiska vetenskaper,
SE-412 96 Göteborg, Sweden\\
E-mail: \url{martin@raum-brothers.eu}\\
Homepage: \url{http://raum-brothers.eu/martin}
}%

\end{document}

